\documentclass[12pt, oneside]{article}   	
\usepackage[left=1in,right=1in,top=1in,bottom=1in]{geometry}                		
\usepackage{subfigure}
\usepackage{graphicx}
\usepackage{amsfonts}
\usepackage{url}
\usepackage{amsmath,amssymb}
\usepackage{enumerate}
\usepackage{amsthm}

\usepackage{mathtools}

\usepackage[colorlinks]{hyperref}

\usepackage{todonotes}

\newtheorem{remark}{Remark}
\newtheorem{lemma}{Lemma}
\newtheorem{theorem}{Theorem}
\newtheorem{corollary}{Corollary}

\theoremstyle{definition}

\newtheorem{defi}{Definition}

\theoremstyle{definition}
\newtheorem{example}{Example}
\newtheorem{assumption}{Assumption}

\newcommand{\EE}{\mathbb{E}}

\newcommand{\R}{{\mathbb R}}
\newcommand{\Z}{{\mathbb Z}}

\newcommand{\E}{\mathbb E}

\newcommand{\Var}{\text{Var}}

\title{Variance of finite difference methods for reaction networks with non-Lipschitz rate functions}

\author{
David F. Anderson\thanks{Department of Mathematics, University of
  Wisconsin, Madison, USA.  anderson@math.wisc.edu.},
\and
Chaojie Yuan\thanks{Department of Mathematics, Indiana University Bloomington, Indiana, USA.  yuan13@iu.edu.
}}

\begin{document}
\maketitle

\begin{abstract}
Parametric sensitivity analysis is a critical component in the study of mathematical models of physical systems.  Due to its simplicity, finite difference methods are used extensively for this analysis in the study of stochastically modeled reaction networks.  Different coupling methods have been proposed to build  finite difference estimators, with the ``split coupling,''  also termed the ``stacked coupling,''  yielding the lowest variance in the vast majority of cases.  Analytical results related to this coupling are sparse, and include an analysis of the variance of the coupled processes under the assumption of globally Lipschitz intensity functions \cite{AndCFD2012}.  Because of the global Lipschitz assumption utilized in \cite{AndCFD2012}, the main result there is only applicable to a small percentage of the models found in the literature, and it was conjectured that similar results should hold for a much wider class of models.  In this paper we demonstrate this conjecture to be true by proving the variance of the coupled processes scales in the desired manner for a large class of non-Lipschitz models.   We further extend the analysis to allow for time dependence in the parameters.   In particular, binary systems with or without time-dependent rate parameters, a class of models that accounts for the vast majority of systems considered in the literature, satisfy the assumptions of our theory.  
\end{abstract}

\section{Introduction}

Stochastic models of biochemical interaction networks are now used ubiquitously in biology, especially cell biology \cite{AK2011, AK2015, elowitz2002stochastic, gillespie1976general, paulsson2004summing, raj2006stochastic, wilkinson2006stochastic}.  The most common such model treats the system as a discrete-space, continuous-time Markov chain with the state vector giving the counts of the constituent species and with transitions for the chain modeled through the possible reactions of the system.  The key system parameters of these models, the kinetic parameters, govern the rates at which the different reactions take place, and computing derivatives of  system outputs with respect to these parameters, i.e.~parametric sensitivity analysis, plays a critical role in many problems related to optimization and uncertainty quantification  \cite{anderson2007modified, AndersonKoyama2015, AndSkubak, AW2015, feng2004optimizing, komorowski2011sensitivity, SAR2013, stelling2004robustness}.  In particular, such derivatives are often utilized in an inner loop in optimization problems, and  their  estimation is one of the key bottlenecks for many computational experiments in systems biology.

Suppose our state vector at time $t$ is denoted by $X^\theta(t)$,  where $\theta$ is a vector of parameters for the model.  For simplicity, assume for now that $\theta$ is one-dimensional.  Suppose also that $f$ is a function and $f(X^\theta(t))$ is some output of  interest.   For example, if we are interested in the total molecular count, we can take $f(x) = \| x\|_1$, whereas if we are interested in  the count of the first species we can take $f(x) = x_1$. Returning to the case of a general function $f$, the sensitivity of $J(\theta)= \EE[ f(X^{\theta} (t) )]$ with respect to $\theta$ is defined  as
\begin{align}\label{eq:7696700}
	J'(\theta)   = \frac{\partial}{\partial \theta} \EE[f(X^{\theta}(t))].
\end{align}

By far the most common  method for the computational estimation of the above derivative  is  finite difference combined with Monte Carlo.  Specifically, we can estimate 
\[
	J'(\theta) \approx \frac{\EE[f(X^{\theta+\varepsilon} (t)) ]  - \EE [f (X^{\theta} (t)  )]}{\varepsilon}  = \EE \left[ \frac{f(X^{\theta+\varepsilon} (t))- f( X^{\theta} (t))}{\varepsilon} \right],
\]
via
\begin{align*}
\EE \left[ \frac{f(X^{\theta+\varepsilon} (t))- f( X^{\theta} (t))}{\varepsilon} \right] \approx D_N(\varepsilon) = \frac{1}{N\varepsilon }\sum_{i=1}^N \left[ f(X_{[ i ]}^{\theta + \varepsilon} (t))- f(X_{[ i ]}^{\theta} (t))\right],
\end{align*}
where $X_{[ i ]}^{\theta}$ denotes the $i^{th}$ realization of $X^{\theta}$. The variance of this estimator is 
\begin{align*}
\Var \left ( D_N (\varepsilon) \right )  & = \frac{1}{N\varepsilon^2} \Var\left(f(X^{\theta+\varepsilon} (t))- f(X^{\theta} (t))\right),
\end{align*}
so long as the pairs $\{(X_{[i]}^{\theta+\varepsilon},X_{[i]}^\theta)\}_{i = 1}^N$ are independent.
For a given $\varepsilon$, there are two ways to decrease the variance of the estimator:
\begin{enumerate}[(i)]
\item increase $N$, or
\item  decrease $\Var\left(f(X^{\theta+\varepsilon} (t))- f(X^{\theta} (t))\right)$.  
\end{enumerate}
Note that if $X^\theta$ and $X^{\theta + \varepsilon}$ are constructed independently, then the variance  is simply \cite{ASV2017}
\begin{equation*}
	\Var\left(f(X^{\theta+\varepsilon} (t))\right) +\Var \left(f(X^{\theta} (t))\right).
\end{equation*}
However, if the processes are \textit{coupled} so that they are constructed on the same probability space but not independent, then
\begin{align*}
 \Var\left(f(X^{\theta+\varepsilon} (t))- f(X^{\theta} (t))\right) &= \Var\left(f(X^{\theta+\varepsilon} (t))\right) +\Var \left(f(X^{\theta} (t))\right) \\
 &\hspace{.2in}- 2\text{Cov}(f(X^{\theta+\varepsilon}(t)), f(X^\theta(t))),
\end{align*}
and it is possible to decrease the variance with little or no extra computational burden.

The two most commonly used couplings for parametric sensitivity analysis of reaction networks are the  \textit{common reaction path coupling} \cite{Khammash2010} and \textit{coupled finite difference}  (CFD)\cite{AndCFD2012}, also termed the \textit{split coupling}, both of which rely on the random time change representation of Kurtz \cite{Kurtz80}. In more recent work, the \textit{stacked coupling} was introduced  \cite{anderson2018low}.  This coupling is similar to the CFD coupling, but utilizes thinning and is more amenable to models with time dependent rate parameters.   A  construction similar to the stacked coupling, and which was developed independently and simultaneously, can be found in \cite{thanh2017efficient}. 

Unfortunately, analytical results related to these couplings are sparse.  In \cite{AndCFD2012}, it was shown that  $\Var\left(f(X^{\theta+\varepsilon} (t))- f(X^{\theta} (t))\right)$ is $O(\varepsilon)$ if the CFD coupling is used and if the intensity functions of the model are globally Lipschitz in both the state variable and the parameter $\theta$.  However, these assumptions are satisfied only if the model stays inside a bounded set (for example, if a conservation relation exists for all species). \textit{Therefore, these assumptions exclude most models found in the literature}.   Similar analytic results related to the stability and convergence of coupled processes were given in section 4 of Engblom's 2014 paper \cite{Engblom2014}.

Our main analytical result, Theorem \ref{mainthm1}, extends the main  result from \cite{AndCFD2012} to a much broader class of models, including all binary networks, which constitutes the vast majority of models utilized in the literature.   Specifically, Theorem \ref{mainthm1} guarantees  that for any model satisfying Assumption \ref{assumpMain}, which are mild regularity conditions satisfied by nearly all models found in the literature, $\E\left[\|X^{\theta+\varepsilon} (t) - X^{\theta} (t)\|_1^p\right] = O(\varepsilon)$, for any $p \ge 1$.  In particular, the main growth condition we assume a model satisfies is that any reaction that increases the net total number of molecules has a rate function that grows linearly.  We note that this condition was also considered in \cite{Engblom2014, GBK2014, gupta2013unbiased, Rathinam2013}.

\textcolor{black}{Finally, we note that finite differences is not the only method for estimating the derivatives in \eqref{eq:7696700}.  In particular, it is often possible to move the derivative inside the expectation and perform a change of measure.  This leads to a Girsanov transformation method \cite{Plyasunov2007,Khammash2010}.  However, in the present setting of reaction networks this method typically has a very large variance, and so is not often utilized. Further,  it is sometimes possible to switch the order of differentiation and expectation in \eqref{eq:7696700} directly, and this leads to pathwise differentiation methods \cite{Glasserman1990}.  However, this method of differentiation is generally not applicable in the present setting since there are commonly ``interruptions'' in the reactions \cite{AW2015}.  However, see \cite{AW2015} for a hybrid pathwise differentiation method that is applicable to the vast majority of models in the literature.}

The remainder of the paper is organized as follows.  In  section \ref{preliminaries}, we will provide  required terminology, the formal mathematical model, and the assumptions on the model we require for our main theorem to hold.  In section \ref{mainresult}, we state and prove our main result, Theorem \ref{mainthm1}.

\section{Mathematical model}\label{preliminaries}

\subsection{Notation}
We use standard notation.   The nonnegative integers, real numbers, and positive real numbers will  be represented by $\Z_{\ge0}$, $\R$, and $\R_{>0}$,  respectively. For $d \ge 2$, we denote by $\Z_{\ge0}^d$, $\R^d$, and  $\R_{>0}^d$ the corresponding $d$-dimensional vector spaces. For  $x\in \Z_{\ge0}^d$ the factorial $x!$ is defined
\begin{align}\label{compfact}
x! = \prod_{i=1}^d x_i !.
\end{align}
For $u,v \in \R^d$, we write $u \ge v$ if $u_i \ge v_i$ for each $i$, we define the inner product as $u\cdot v = \sum_{i=1}^d u_i v_i$, and the $\ell_1$  norm of $u$ is 
\begin{align*}
\| u \|_1 = \sum_{i=1}^d |u_i |.
\end{align*}
For  $A \subset \R$ and $x\in \R$ we denote the indicator function $1_A(x)$  by
\begin{align*}
1_{A} (x) = \begin{cases}
1,& x\in A\\
0,& \text{ otherwise. }
\end{cases}
\end{align*}

\subsection{Reaction networks and stochastic mass-action kinetics}
\label{sec:terminology}
We begin with the definition of a reaction network.  See \cite{AK2015} for any necessary background.

\begin{defi}
A reaction networks is a triple  $\{\mathcal{S},\mathcal{C},\mathcal{R}\}$ where
\begin{itemize}
\item $\mathcal{S} = \{A_1,\dots,A_d\}$ is a set of \emph{species}.
\item $\mathcal{C}$, the \emph{complexes},  is a non-empty, finite set of linear combinations of the species over the non-negative integers.  We make the usual abuse of notation by  corresponding $y\in \mathcal{C}$ with the vector in $\Z^d_{\ge 0}$ whose $i$th component gives the integer count of species $A_i$ in complex $y$.
\item $\mathcal{R}$, the \emph{reactions}, is a  subset of $\mathcal{C} \times \mathcal{C}$.  We typically denote $(y,y')\in \mathcal{R}$ by $y \to y'$ and will denote by $K$ the cardinality of $\mathcal{R}$.  

After enumerating the reactions, we denote the $k^{th}$ reaction by $y_k \to y_k' \in \mathcal{R}$,  we denote its \textit{reaction vector} by $\zeta_k = y_k'- y_k \in \Z^d$, and we will call $y_k$ and $y_k^\prime$ the source and product complexes for the $k$th reaction, respectively. 
\end{itemize}
\end{defi}

The typical stochastic model of a reaction network treats the system as a continuous-time Markov chain with state space $\Z^d_{\ge 0}$,  where $X(t)\in \Z^d_{\ge 0}$ gives the counts of the different chemical species at time $t\ge 0$. Transitions for the model are determined by the reactions and their intensity functions. Specifically, for the $k^{th}$ reaction, we let $\lambda_k: \Z^d \times [0,\infty) \rightarrow \R_{\ge 0}$  be the corresponding intensity function.  Then the transition rate from state $x$ to $x'$ is
\begin{align}\label{TransitionRates}
Q(x,x',t) = \sum_{k: \zeta_k = x' - x} \lambda_k(x,t),
\end{align}
where the sum is over those reactions with reaction vector equal to $x'-x \in \Z^d_{\ge 0}$.
The most common choice of intensity function is given by stochastic mass-action kinetics, in which case
\begin{align}\label{stochasticMA}
\lambda_k(x)  = \kappa_{k}  \frac{x!}{(x- y_k)! } 1_{\{x\geq y_k\}},
\end{align}
where $y_k$ is the source complex.  Because this is the most common choice of rate function, reactions with source complex of the form $\emptyset$ are termed \textit{zeroth order reactions} and reactions with source complex $y_k = e_i$ for some $i \in \{1,\dots,d\}$ are termed \textit{first order reactions}.

Note that mass-action kinetics does not  depend upon time, which is why the $t$ dependence has been dropped.
When stochastic mass-action kinetics is used, the term $\kappa_k$ is called a \textit{rate constant}, and is typically placed next to the reaction arrow in the reaction diagram (see Examples \ref{example:first},  \ref{example:dimer}, and \ref{example:second} below). 

The most commonly used time-dependent  intensity function  is only a slight modification of \eqref{stochasticMA}, and assumes that the rate parameters are functions of time,
\begin{equation}\label{TimeDependentMassAction}
	\lambda_k(x,t) = \kappa_k(t) \frac{x!}{(x-y_k)!} 1_{\{x\ge y_k\}}.
\end{equation}

A few examples illustrate the model.

\begin{example}\label{example:first}
Consider the following reaction network with stochastic mass-action kinetics
\begin{align*}
	A &\overset{2}\to  2A \qquad 2A \overset{\theta}\to A.
\end{align*}
Because we are assuming mass-action kinetics, the intensity functions are
\[
	\lambda_1(x) = 2x \quad \text{and}\quad \lambda_2(x)= \theta x (x-1), \quad \text{for}\quad x \in \Z_{\ge 0}.
\]
Neither intensity is a function of time,  so there is no $t$-dependence in the notation.
\hfill $\square$ 
\end{example}

\begin{example}[Intracellular viral kinetics]
Consider the following model given in \cite{gupta2013unbiased}. This model consists of 4 species: the viral template (T), the viral genome (G), the viral structural protein (S), and the virus (V). There are 6 reactions
\begin{align*}
	T &\overset{k_1}\to  T+G \qquad G \overset{k_2}\to  T \qquad T \overset{k_3}\to  T+ S \qquad T \overset{k_4}\to \emptyset \qquad S \overset{k_5}\to \emptyset \qquad G+S \overset{k_6}\to V.
\end{align*}
Assuming mass-action kinetics, and after ordering the species as $T,G,S,V$, we have
\[
\lambda_1(x) = \kappa_1 x_1,\hspace{.051in} \lambda_2(x) = \kappa_2 x_2, \hspace{.051in} \lambda_3(x) = \kappa_3 x_1, \hspace{.051in} \lambda_4(x) = \kappa_4 x_1, \hspace{.051in} \lambda_5(x) = \kappa_5 x_3, \hspace{.051in} \lambda_6(x)= \kappa_6 x_2x_3.
\]
Notice that $G$ and $S$ can be created with the presence of only one copy of a viral template.  Hence, the state space will be unbounded. 
\hfill $\square$ \label{example:dimer}
\end{example}

The following example  incorporates circadian rhythm into our intensity function by assuming time dependence of a rate parameter. 

\begin{example}\label{example:second}
Consider the following standard network for gene transcription and translation,
\begin{align*}
	\emptyset &\overset{\lambda_1(t)}\to  M  \qquad	M \overset{100}\to  M + P\qquad M \overset{1}{\to} \emptyset \qquad P \overset{1}{\to} \emptyset.
\end{align*}
We assume the last three reactions have mass-action kinetics \eqref{stochasticMA} with associated rate constants given above their reaction arrow.  However, we assume the intensity function for the first reaction follows \eqref{stochasticMA}.  In particular,
\[
	\lambda_1(t)= \kappa_1(t) = 60+15\sin\left(\tfrac{2\pi t}{24}\right).
\]
Using such an intensity function for the first reaction allows us to model the system going through dark-light cycles that oscillate over a 24 hour time-period.\hfill $\square$ 
\end{example}

\subsection{Representations for the process}
We provide two representations for the process detailed in the previous section.   See \cite{AK2011,AK2015} for more details on either. The first is the random time change representation of Thomas Kurtz \cite{Kurtz80}, in which $X$ is the solution to the following stochastic equation
\begin{align}\label{RTC}
X(t) = X(0) + \sum_{k=1}^K Y_k \left(\int_0^t \lambda_k(X(s),s) ds \right) \zeta_k,
\end{align}
where the $\{Y_k\}_{k=1}^K$ are independent unit-rate Poisson processes.

The second  representation arrises by letting $X$ be the solution to the following equation
\begin{align}\label{PPP}
X(t) = X(0) + \sum_{k=1}^K\zeta_k \int_{[0,t]\times [0,\infty)}  1_{[q_{k-1}(s-) , q_{k-1} (s-)+ \lambda_k(X(s-),s-) )} (x)  N(ds\times dx) ,
\end{align}
where $N$ is a unit-rate space-time Poisson point process and $q_0(s) = 0$ for all $s \ge 0$ and for $k \in \{1,\dots,K\}$,
\begin{align*}
q_{k} (s)   = \sum_{\ell =1}^k \lambda_{\ell} (X(s),s).
\end{align*}

\subsection{Coupled processes and assumptions on the model}

We detail a number of couplings already found in the literature.

 The first coupling of  $(X^{\theta}, X^{\theta+\varepsilon})$ we detail, usually termed the \textit{common reaction path coupling} \cite{Khammash2010}, arises from using the same choice of Poisson processes in the representation \eqref{RTC}.  Specifically, 
\begin{align}\label{crp}
\begin{split}
X^{\theta+\varepsilon}(t) &= X^{\theta+\varepsilon}(0) + \sum_{k=1}^K Y_k \left(\int_0^t \lambda_k^{\theta+\varepsilon}(X^{\theta+\varepsilon}(s),s) ds \right) \zeta_k \\
X^{\theta}(t) &= X^{\theta}(0)+ \sum_{k=1}^K Y_k \left(\int_0^t \lambda_k^{\theta}(X^{\theta}(s),s) ds \right) \zeta_k, 
\end{split}
\end{align}
where $\{Y_k\}_{k=1}^K$ is a sequence of independent unit-rate Poisson processes.

The second coupling of  $(X^{\theta}, X^{\theta+\varepsilon})$ we detail, termed the \textit{stacked coupling} \cite{anderson2018low},  utilizes a space-time Poisson point process as its randomness, as in \eqref{PPP}. Specifically, 
\begin{align}\label{stackedcoupling}
\begin{split}
X^{\theta+\varepsilon}(t) & = X^{\theta+\varepsilon}(0) + \sum_{k=1}^K \zeta_k \int_{[0,t] \times [0,\infty)} 1_{ [q_{k-1}^{\theta,\varepsilon}(s-) , q_{k-1}^{\theta,\varepsilon} (s-)+ \lambda_k^{\theta+\varepsilon}(X^{\theta+\varepsilon}(s-),s-) )} (x) N(ds\times dx) \\
X^{\theta} (t) & = X^{\theta} (0) + \sum_{k=1}^K \zeta_k \int_{[0,t] \times [0,\infty)} 1_{ [q_{k-1}^{\theta,\varepsilon}(s-) , q_{k-1}^{\theta,\varepsilon} (s-)+ \lambda_k^{\theta}(X^{\theta}(s-),s-) )} (x) N(ds\times dx),
\end{split}
\end{align}
where 
\begin{align}\label{notations}
\begin{split}
\overline{\lambda}^{\, \theta,\varepsilon}_{\ell} (s)  & = \max \{ \lambda_\ell^{\theta+\varepsilon}\left (X^{\theta+\varepsilon}(s),s \right) , \lambda_\ell^{\theta}\left (X^{\theta}(s),s \right) \} \\
q_{k}^{\theta,\varepsilon} (s)  & = \sum_{\ell=1}^k \overline{\lambda}^{\, \theta,\varepsilon}_{\ell} (s) \quad \quad \text{      with }   \quad \quad  q_0^{\theta,\varepsilon}(s) = 0.
\end{split}
\end{align}

\textcolor{black}{
\begin{remark}\label{Re:stackedcoupling}
One key advantage of the stacked coupling is that whenever both processes jump at the same time, they must jump via the same reaction channel. That is, when the $k$th reaction channel is chosen to fire for $X^{\theta}$,  then $X^{\theta+\varepsilon}$ can either update through the $k$th reaction channel or stay put, and vice versa. 
\end{remark}}

For future reference, we define
\begin{align}\label{simulationrep}
N_{q_K}^{\theta,\varepsilon}(t) &= \int_{[0,t] \times [0,\infty)} 1_{ [0, q_{K}^{\theta,\varepsilon} (s-) )} (x) N(ds\times dx),
\end{align}
to be the number of points of the point process $N$ in our region of interest up to time $t$.

If the intensity functions are independent of time, then the stacked coupling \eqref{stackedcoupling} is equivalent to the coupled finite difference coupling (CFD), also termed the split coupling \cite{AndCFD2012},
\begin{align}\label{eq:CFD}
\begin{split}
X^{\theta+\varepsilon}(t) =   X^{\theta+\varepsilon}(0) &+ \sum_{k=1}^K \zeta_k Y_k^1 \left (\int_0^t \lambda_k^{\theta+\varepsilon} (X^{\theta+\varepsilon}(s),s) \wedge\lambda_k^{\theta} (X^{\theta}(s),s)  ds \right) \\
& + \sum_{k=1}^K \zeta_k Y_k^2 \left(\int_0^t \left( \lambda_k^{\theta+\varepsilon}(X^{\theta+\varepsilon}(s),s)  - \lambda_k^{\theta+\varepsilon} (X^{\theta+\varepsilon}(s),s) \wedge\lambda_k^{\theta} (X^{\theta}(s),s) \right)  ds \right) \\
X^{\theta}(t)  =   X^{\theta}(0) &+ \sum_{k=1}^K \zeta_k Y_k^1 \left (\int_0^t \lambda_k^{\theta+\varepsilon} (X^{\theta+\varepsilon}(s),s) \wedge\lambda_k^{\theta} (X^{\theta}(s),s)  ds \right) \\
& + \sum_{k=1}^K \zeta_k Y_k^3 \left(\int_0^t \left(  \lambda_k^{\theta} (X^{\theta}(s),s)- \lambda_k^{\theta+\varepsilon} (X^{\theta+\varepsilon}(s),s) \wedge\lambda_k^{\theta} (X^{\theta}(s),s) \right)  ds \right),
\end{split}
\end{align}
where $\{Y_{k}^1,Y_k^2, Y_k^3\}_{k = 1}^K$ are independent unit-rate Poisson processes.

As noted in the introduction, analytical results related to these couplings are sparse.  However, the following is proven in \cite{AndCFD2012}.

\begin{theorem}\label{thm:CFDvariance}
Let $( \mathcal{S},\mathcal{C}, \mathcal{R})$ be a reaction network with $d$ species and $K$ reactions. Let  $\{ X^{\theta} \}$ be a family of stochastic models whose intensity functions $\lambda_k^{\theta}(x,t)$ are parametrized by $\theta\in \R_{>0}^m$. Suppose $(X^{\theta+\varepsilon}, X^{\theta})$ are coupled as in \eqref{eq:CFD} and there exist constants $K_1, K_2$ such that for all $x,y\in \Z_{\ge 0}^d$
\begin{align}\label{ass:CFDvariance1}
\sum_{k=1}^K \left|\lambda_k^{\theta}(x) - \lambda_k^{\theta}(y) \right | \leq K_1 |x-y| 
\end{align}
and for all $x\in  \Z_{\ge 0}^d$, and all $\varepsilon <1 $ 
\begin{align}\label{ass:CFDvariance2}
\sum_{k=1}^K \left|\lambda_k^{\theta+\varepsilon}(x) - \lambda_k^{\theta}(x) \right | \leq K_2 \varepsilon. 
\end{align}
Then, for any $T> 0$ and any function $f: \R^d \to \R$ that is $C^1$ (bounded first derivative),  there is a  $C_{T,f}>0$ for which  
\begin{align*}
\EE \left[\sup_{t\leq T} \left( f(X^{\theta+\varepsilon}(t)) - f(X^{\theta}(t))\right)^2 \right] \leq C_{T,f} \varepsilon.
\end{align*}
\end{theorem}
Notice that \eqref{ass:CFDvariance1} implies the intensity functions are globally Lipschitz in the state variable and \eqref{ass:CFDvariance2} implies the intensity functions are globally Lipschitz in the parameter, these assumptions are highly restrictive and most models in the literature are excluded. 

\begin{example}
Consider Example \ref{example:first} and recall the reaction intensity for the second reaction is
\[
	\lambda_2(x,t)= \theta x (x-1).
\]
Notice that the function is only locally Lipschitz in $x$ and $\theta$ (since the state space is not bounded), thus Theorem \ref{thm:CFDvariance} is not  applicable.
\hfill $\square$ 
\end{example}

\begin{example}
Consider the intracellular viral kinetics model introduced in Example \ref{example:dimer}. Notice that the last reaction is binary, hence $\lambda_6(x)=\kappa_6 x_2x_3$ is only locally Lipschitz in $x$. Moreover,  the viral genome (G) and the viral structural protein (S) can be created by the first and third reaction without consuming any other species. Hence $x_2$ and $x_3$ can not be bounded and $\lambda_2, \lambda_5, \lambda_6$ are only locally Lipschitz in terms of rate constants. Thus, Theorem \ref{thm:CFDvariance} is not  applicable. 
\hfill $\square$ \label{example:dimer222}
\end{example}

We will extend Theorem \ref{thm:CFDvariance} to a wider class of models.  However, care must be taken not to try to extend too broadly.  For example, models such as $2A \to 3A$ are explosive, and their expectation and variance are not finite.   With this example in mind, we denote by $\mathcal{P} = \{k\in \{1,\dots,K\}: \zeta_k\cdot \vec{1} >0\}$,  the indices of the reactions that have a positive net effect on the count of the total population, and give our main modeling assumption below.

\begin{assumption}\label{assumpMain}
A family of rate functions, $\{\lambda_k^\theta\}$, parameterized by $\theta \in \R^\gamma_{>0}$  satisfies this assumption if there is an $\overline C>0$ and a compact set $\Theta \subset \R^\gamma_{>0}$ for which the following conditions hold.
\begin{enumerate}[(1)]

\item\label{assump1} (\textbf{Linear growth for $\mathcal{P}$})  For any $k\in \mathcal{P}$ 
\begin{align*}
	\sup_{\theta \in \Theta} \sup_{t \in \R_{>0}} \lambda_k^{\theta} (x,t)  \leq \overline C (1+\| x \|_1).
\end{align*}

\item\label{growth}  (\textbf{Polynomial growth for $\mathcal{R}$}) For any $k\in \{1,2,...,K\} $, there is an integer $p \geq 1$ for which 
\begin{align*}
	\sup_{\theta \in  \Theta} \sup_{t \in \R_{>0}}  \lambda_k^{\theta} (x,t)  \leq \overline C (1+\| x \|_1^p)
\end{align*}

\item \label{coupling} For $\|\varepsilon\|_1> 0$ small enough, 
\begin{align*}
\sup_{\theta \in \Theta} \sup_{x\in \Z_{\ge 0}^d}\frac{\sum_{k=1}^K |\lambda_k^{\theta+\varepsilon}(x,s)-\lambda_k^{\theta}(x,s) |}{\sum_{k=1}^K \max\{\lambda_k^{\theta+\varepsilon}(x,s),\lambda_k^{\theta}(x,s)\}} \leq \overline C \|\varepsilon\|_1.
\end{align*}
\end{enumerate}
\end{assumption}

\begin{remark}\label{assump1_noset}
We will sometimes work with sets of rate functions with a particular choice of rate constants, and these will be denoted by $\{\lambda_k\}$.  In this case, the set $\Theta$ is taken to be a single point and the criterion $\sup_{\theta\in\Theta}$ can be dropped from each of the three conditions.
\end{remark}

We   note that  condition (1) in Assumption \ref{assumpMain} was also utilized in \cite{Engblom2014, GBK2014, gupta2013unbiased, Rathinam2013}.

The next lemma shows that stochastic mass-action kinetics \eqref{stochasticMA} satisfies conditions  (\ref{growth}) and (\ref{coupling})  of Assumption \ref{assumpMain}.    Hence, any model with this choice of kinetics satisfies Assumption \ref{assumpMain} if only first or zeroth order reactions generate a net gain in total molecule counts.  \textcolor{black}{In particular, binary systems, a class of models that accounts for the vast majority of systems considered in the literature, satisfy this assumption, since a reaction with a binary source complex will either maintain or decrease total molecule counts when it occurs.   For an example of a model that does \textit{not} satisfy condition (1) when mass-action kinetics is utilized, simply consider $2A \rightleftarrows 3A$. In particular, the reaction $2A \to 3A$ violates the assumption.}

\begin{lemma}\label{massactionassump}
Let $( \mathcal{S},\mathcal{C}, \mathcal{R})$ be a reaction network with $d$ species and $K$ reactions.  Assume the intensity functions are given by stochastic mass-action kinetics \eqref{stochasticMA}.  Let $\Theta \subset \R^K_{>0}$ be a compact set, where we correspond each $\theta = (\kappa_1, \kappa_2,..., \kappa_K)\in \Theta$ with a choice of mass-action rate constants. Then, conditions (\ref{growth}) and (\ref{coupling}) of Assumption \ref{assumpMain} are  satisfied. 
\end{lemma}

\begin{proof}
Condition \eqref{growth} of Assumption \ref{assumpMain} follows easily since we have  
\begin{align*}
\lambda_k(x)  &= \kappa_{k} \frac{x!}{(x- y_k)! } 1_{\{x\geq y_k\}} \leq \kappa_k \prod_{i=1}^d \left[ x_i (x_i-1) \ldots (x_i - y_{ki} +1) \right] \leq \kappa_k \prod_{i=1}^d x_i^{y_{ki}} \leq \kappa_k \|x\|_1^{\|y_k\|_1},
\end{align*}
where we take $0^0 = 1$, and since $\theta \in \Theta$, which is compact.

To verify condition \eqref{coupling} of  Assumption \ref{assumpMain}, notice that 
\begin{align*}
|\lambda_k^{\theta+\varepsilon}(x,s)-\lambda_k^{\theta}(x,s)|  & = |\varepsilon_k| \frac{x!}{(x- y_k)! } 1_{\{x\geq y_k\}}, \quad \text{and} \\
 \max\{\lambda_k^{\theta+\varepsilon}(x,s),\lambda_k^{\theta}(x,s) \} & = \max \{\kappa_k + \varepsilon_k, \kappa_k\}\frac{x!}{(x- y_k)! } 1_{\{x\geq y_k\}},
\end{align*}
where we choose $\varepsilon_k$  small enough so that $\kappa_k + \varepsilon_k >0$.
Hence 
\begin{align*}
\max_{x\in \Z_{\ge 0}^d}&\frac{\sum_{k=1}^K |\lambda_k^{\theta+\varepsilon}(x,s)-\lambda_k^{\theta}(x,s) |}{\sum_{k=1}^K \max\{\lambda_k^{\theta+\varepsilon}(x,s),\lambda_k^{\theta}(x,s)\}} 
 = \max_{x\in \Z_{\ge 0}^d} \frac{\sum_{k=1}^K |\varepsilon_k| \frac{x!}{(x- y_k)! } 1_{\{x\geq y_k\}} }{\sum_{k=1}^K  \max \{\kappa_k + \varepsilon_k, \kappa_k\}\frac{x!}{(x- y_k)! } 1_{\{x\geq y_k\}} } \\ 
&\hspace{1.2in} \leq \max_{x\in \Z_{\ge 0}^d} \max_k \left \{ \frac{|\varepsilon_k|}{\max \{\kappa_k + \varepsilon_k, \kappa_k\}} \right\} =  \max_k \left \{ \frac{|\varepsilon_k|}{\max \{\kappa_k + \varepsilon_k, \kappa_k\}} \right\}.
\end{align*}
For the last inequality, we used the fact that
\begin{align*}
a_k\leq c b_k \qquad \Rightarrow \qquad \frac{\sum_{k=1}^K a_k}{\sum_{k=1}^K b_k} \leq c \frac{\sum_{k=1}^K b_k}{\sum_{k=1}^K b_k} = c.
\end{align*}
Hence, again by the compactness of $\Theta \in \R^K_{>0}$, condition \eqref{coupling} of  Assumption \ref{assumpMain} holds so long as we choose $\|\varepsilon\|_1$ small enough, .
\end{proof}

\section{Main results}\label{mainresult}

In this section, we prove our main result, stated below.

\begin{theorem}\label{mainthm1}
Let $( \mathcal{S},\mathcal{C}, \mathcal{R})$ be a reaction network with $d$ species and $K$ reactions. Let  $\{ X^{\theta} \}$ be a family of stochastic models whose associated intensity functions $\{\lambda_k^{\theta}\}$ are parametrized by $\theta\in \Theta \subset \R_{>0}^\gamma$,  for some $\gamma \in \Z_{>0}$, where $\Theta$ is compact.  Suppose that Assumption \ref{assumpMain} holds and that $(X^{\theta}, X^{\theta+\varepsilon})$ are coupled via the stacked coupling (\ref{stackedcoupling}) with $X^{\theta}(0) = X^{\theta + \varepsilon}(0) = x_0 \in \Z^d_{\ge 0}$. Then, for any $\theta$ in the interior of  $\Theta$  and any  $r\ge 1$, there is a $C_{r,t}>0$  and $\overline \varepsilon \in \R_{> 0}$ so that 
\begin{equation}\label{eq:mainbound}
	\EE \left[\| X^{\theta+\varepsilon} (t) -X^{\theta}(t) \|_1^r \right] \leq C_{r,t} \| \varepsilon\|_1,
	\end{equation}
when $ \| \varepsilon \|_1 \le \overline \varepsilon$.
\end{theorem}

The condition that $\theta$ be in the interior of $\Theta$ is not vital.  In particular, if $\theta$ were on the boundary of $\Theta$, then a larger $\widetilde \Theta \supset \Theta$ could be chosen.  \textcolor{black}{Note also that the choice of $\overline \varepsilon$ can be used to guarantee that $\theta + \varepsilon \in \Theta$.}

Note that an immediate corollary of Theorem \ref{mainthm1}, acquired by taking $r = 2$, is the following.
\begin{corollary}\label{maincor}
Let $( \mathcal{S},\mathcal{C}, \mathcal{R})$ be a reaction network with $d$ species and $K$ reactions. Let  $\{ X^{\theta} \}$ be a family of stochastic models whose associated intensity functions $\{\lambda_k^{\theta}\}$ are parametrized by $
\theta\in \R_{>0}^\gamma$,  for some $\gamma \in \Z_{>0}$.  Suppose that Assumption \ref{assumpMain} holds and that $(X^{\theta}, X^{\theta+\varepsilon})$ are coupled via the stacked coupling (\ref{stackedcoupling}) with $X^{\theta}(0) = X^{\theta + \varepsilon}(0) = x_0 \in \Z^d_{\ge 0}$.    Finally, let $f: \R^d \to \R$ be  Lipschitz.  Then there is a $C_{T,f}>0$ and  a $\overline \varepsilon \in \R_{> 0}$ so that  
$$
Var \left( f(X^{\theta+\varepsilon} (t)) - f(X^{\theta}(t)) \right) \leq C_{T,f} \| \varepsilon\|_1, 
$$
when $ \| \varepsilon \|_1 \le \overline \varepsilon$.
\end{corollary}

The rest of the section is organized as follows. In section \ref{geometricgrowthbound}, we will discuss the key component of the proof, which establishes an upper bound for the rate of  growth  of the process. In section \ref{morelemmas}, we will then introduce some technical lemmas.  The pieces will then be put together in section \ref{proofmain}.

\subsection{Growth bound of the stochastic process}\label{geometricgrowthbound}
Let $\tau_m$ be the first time the process leaves the $L^1$ ball with radius $m$.  We are interested in the decay rate of $P(\tau_m \leq t)$, as $m \to \infty$, and will show in Lemma \ref{boundednessAssumption} that it decays exponentially so long as the model satisfies  condition \eqref{assump1} of Assumption \ref{assumpMain}.  Moreover, by Lemma \ref{lemma3}, the bound is sharp.   We note that Lemma \ref{boundednessAssumption} is similar to a result in \cite{gupta2013unbiased}, where, under the same assumptions, the decay was shown to be polynomial in $m$.

\begin{lemma}\label{boundednessAssumption}
Let $( \mathcal{S},\mathcal{C}, \mathcal{R})$ be a reaction network with $d$ species and $K$ reactions.    Let  $\{ X^{\theta} \}$ be a family of stochastic models whose associated intensity functions $\{\lambda_k^{\theta}\}$ are parametrized by $\theta\in \Theta \subset \R_{>0}^\gamma$,  for some $\gamma \in \Z_{>0}$, where $\Theta$ is compact.   
Assume $\{\lambda_k^\theta\}$ satisfies condition \eqref{assump1} of Assumption \ref{assumpMain}. 
 Define
 \[
	 \tau_m^\theta = \inf \left\{t\geq 0 : \| X^\theta(t)\|_1 = \sum_{i=1}^d X_i^\theta(t)  \geq m \right\}.
\]
Then, for any initial condition $x_0\in \Z_{\ge 0}^d$\textcolor{black}{, any $\theta \in \Theta$, and any $t\geq 0$}, there exist constants $C >0 $ and $\delta  \in (0,1)$, both of which are independent of $\theta$, such that 
\begin{align}\label{probgeo}
	P(\tau_m^\theta \leq t) \leq C \delta^m, 
\end{align}
for all $m \in \Z_{> 0}$ large enough. 
\end{lemma}

In the proof, we will show that one possible choice of $C$ and $\delta$ is 
\begin{equation}
C = (1-e^{-c t \overline CK})^{-1-\frac{\|x_0\|_1}{\ell}} \text{              and         }  \delta= \left(1-e^{-c t \overline CK} \right)^{\frac{1}{\ell}},
\end{equation}
where $\overline C$ comes from Assumption \ref{assumpMain}, $c = \max \{ \|x_0\|_1+1,\ell\}$ and $\ell  = \max_k\{ \zeta_k \cdot \vec 1\}$.  
 By Remark \ref{assump1_noset}, we may apply Lemma \ref{boundednessAssumption} when only a particular model (with a particular choice of rate constants) is being considered.  In this case, we take $\Theta$ to be a single point.

We prove Lemma \ref{boundednessAssumption} at the end of the section.  The proof will proceed by comparing the model of interest with a particular linear (i.e., first order) model.  We therefore begin with a sequence of lemmas related to linear models.  The first concerns a pure birth-process with a particular choice of rate constants.

\begin{lemma} \label{lemma3}
Suppose $X$ satisfies
\begin{align*}
X(t) & = 1 + Y \left( \kappa \int_0^t X(s)ds \right),
\end{align*}
where $Y$ is a unit-rate Poisson process and $\kappa>0$. For any integer $M \ge 2$, define $\tau_M =  \inf \{t\geq 0: X(t)  \geq M \}$. 
Then 
\[
P(\tau_M \leq t) =  (1-e^{-\kappa t})^{M-1}.
\]
\end{lemma}

\begin{proof}
We first prove the results when $\kappa=1$. Suppose the unit-rate Poisson process $Y$ has holding times $e_1, e_2, \dots$, where $e_i$ are unit exponentials. Then 
\begin{align*}
\tau_M &= \frac{e_1}{1} + \frac{e_2}{2} + \frac{e_3}{3} +\ldots \frac{e_{M-1}}{M-1}= \sum_{i=1}^{M-1} \frac{e_i}{i},
\end{align*}
Notice that $\tau_M$ is the sum of exponential random variables with distinct parameters $\rho_i = i$. Thus $\tau_M$ is hypoexponentially distributed with density (see  \cite{bibinger2013notes}):
\begin{equation*}
f_{\tau_M} (t) = \sum_{i=1}^{M-1} \rho_i e^{-\rho_i t }\left (\prod_{j\neq i}^{M-1} \frac{\rho_j}{ \rho_j - \rho_i}\right).
\end{equation*}
Thus, 
\begin{align*}
P(\tau_M \leq t) & =  \int_0^t f_{\tau_M} (s) ds = \int_0^t \sum_{i=1}^{M-1} \rho_i e^{-\rho_i s } \left(\prod_{j\neq i}^{M-1} \frac{\rho_j}{ \rho_j - \rho_i} \right) ds \\
& =  \sum_{i=1}^{M-1} \left[ \prod_{j\neq i}^{M-1} \frac{\rho_j}{ \rho_j - \rho_i}\right]   \int_0^t \rho_i e^{-\rho_i s } ds =  \sum_{i=1}^{M-1} \left(\prod_{j\neq i}^{M-1} \frac{\rho_j}{ \rho_j - \rho_i} \right) (1-e^{-\rho_i t}) \\
& = \sum_{i=1}^{M-1} \left(\prod_{j\neq i}^{M-1} \frac{\rho_j}{ \rho_j - \rho_i}\right) - \sum_{i=1}^{M-1} \left(\prod_{j\neq i}^{M-1} \frac{\rho_j}{ \rho_j - \rho_i}\right)e^{-\rho_i t} .
\end{align*}
To simplify we note that 
\begin{align*}
\prod_{j\neq i}^{M-1} \frac{\rho_j}{ \rho_j - \rho_i}  & = \prod_{j\neq i}^{M-1} \frac{j}{j - i} = \frac{1}{1-i} \frac{2}{2-i}  \cdots \frac{i-1}{-1} \frac{i+1}{1} \frac{i+2}{2} \cdots \frac{M-1}{M-1-i}  \\
&  = (-1)^{i-1} \frac{1}{i-1} \frac{2}{i-2}\cdots \frac{i-1}{1} \frac{i+1}{1} \frac{i+2}{2} \cdots \frac{M-1}{M-1-i} \\
& = (-1)^{i-1} \frac{(M-1)!}{(M-1-i)! i!} =- (-1)^{i} \binom{M-1}{i}.
\end{align*}
Thus, 
\begin{align}
\begin{split}
P(\tau_M \leq t) 
&  = - \sum_{i=1}^{M-1}(-1)^{i} \binom{M-1}{i}   + \sum_{i=1}^{M-1} (-1)^{i} \binom{M-1}{i}e^{-i t} \\
&  = - \sum_{i=0}^{M-1}(-1)^{i} \binom{M-1}{i}   + \sum_{i=0}^{M-1} (-1)^{i} \binom{M-1}{i}e^{-i t} \\
&= (1-e^{-t})^{M-1},
\end{split}
\label{eq:876978698}
\end{align}
where the third equality follows since the $i = 0$ terms cancel out, and the fourth  equality follows by applying the binomial theorem twice.

The result is therefore shown when $\kappa = 1$.  When $\kappa \ne 1$,  we have
\begin{align*}
\tau_M &= \frac{e_1}{\kappa} + \frac{e_2}{2\kappa} + \frac{e_3}{3\kappa} +\ldots + \frac{e_{M-1}}{(M-1)\kappa} = \frac{1}{\kappa} \sum_{i=1}^{M-1} \frac{e_i}{i} 
\end{align*}
and, simply by scaling, 
\begin{align*}
P(\tau_M \leq  t ) = P \left( \sum_{i=1}^{M-1} \frac{e_i}{i}  \leq \kappa t \right) = (1-e^{-\kappa t})^{M-1},
\end{align*}
where we used \eqref{eq:876978698}.
\end{proof}

In the next lemma we will consider models with a general positive initial condition and with linear intensity functions but larger jumps sizes (denoted by $\ell$).  In particular, the model below can be thought of as arising from the reaction network \textcolor{black}{$A \overset{\kappa}\to (\ell+1) A$}.

\begin{lemma}\label{lemmalemma}
Suppose $X$ satisfies
\begin{align}\label{eq:representation12121}
X(t) & =  x_0+ \ell Y \left(\kappa \int_0^t X(s) ds\right) 
\end{align}
where $Y$ is a unit-rate Poisson process and $\kappa>0$. For any integer $M  > x_0$, define $\tau_M =  \inf \{t\geq 0: X(t)  \geq M \}$. Then for $M \ge \ell + x_0$,
\[
P(\tau_M \leq t) \leq C \delta^M
\]
for some constant $C = C(t,\ell ,\kappa,x_0) >0$ and $\delta = \delta(t,\ell,\kappa, x_0) \in (0,1)$.
\end{lemma}
\begin{proof}
Let $Y$ be the Poisson process in \eqref{eq:representation12121} and define $Z$ to be the solution to
\begin{align*}
Z(t) & =  1 +  Y \left(\kappa \int_0^t Z(s) ds \right).
\end{align*}
Note that $Z$ is using the same Poisson process as $X$.  However, the two processes are different as $Z$ has an initial condition of $Z(0) = 1$ and  only jumps by size 1. In particular, the process $Z$ satisfies the assumptions of Lemma \ref{lemma3}.
By Lemma \ref{lemma3}, if we denote $\mu_m = \inf \{t\geq 0: Z(t) \geq m \}$, we have
\begin{align}\label{eq:dfhsdf}
P(\mu_m \leq t) =  (1-e^{-\kappa t})^{m-1}.
\end{align}
Denote the unit-rate exponential holding times of the Poisson process $Y$ by  $e_1, e_2, \dots$.  Then, for $m \ge 2$, 
\begin{align*}
\tau_{(m-1)\ell +x_0} & =  \frac{1}{\kappa} \left(\frac{e_1}{x_0} + \frac{e_2}{x_0+\ell} + \frac{e_3}{x_0+2\ell} + \ldots + \frac{e_{m-1}}{x_0+(m-2)\ell} \right) = \frac{1}{\kappa} \sum_{i=1}^{m-1} \frac{e_i}{x_0+(i-1)\ell} \\
\mu_m &=  \frac{1}{\kappa} \left(\frac{e_1}{1} + \frac{e_2}{2} + \frac{e_3}{3} +\ldots + \frac{e_{m-1}}{m-1} \right)  = \frac{1}{\kappa} \sum_{i=1}^{m-1} \frac{e_i}{i}.
\end{align*}
Hence if we let $ c = \max \{x_0,\ell\}$,
\begin{align*}
\tau_{(m-1)\ell+x_0} & = \frac{1}{\kappa} \sum_{i=1}^{m-1} \frac{e_i}{x_0+(i-1)\ell} \geq  \frac{1}{\kappa} \sum_{i=1}^{m-1} \frac{e_i}{c+(i-1)c} = \frac{1}{c} \mu_{m}.
\end{align*}
Hence,  by \eqref{eq:dfhsdf}
\begin{equation*}
P(\tau_{(m-1)\ell+x_0}  \leq t ) \leq P(\mu_{m}  \leq c t) = (1-e^{-c t \kappa})^{m-1}.
\end{equation*}
Thus, for $M \ge \ell + x_0$,   we may take $m = \lfloor\frac{M-x_0}{\ell}\rfloor + 1$, which must be greater than or equal to 2, and conclude
\begin{equation*}
P(\tau_{M} \leq  t )  \leq P(\tau_{(m-1)\ell+x_0} \leq t )  \leq  (1-e^{-c t \kappa})^{m-1}   \leq  (1-e^{-c t \kappa})^{\frac{M-x_0}{\ell}-1}  =  C \delta^M,
\end{equation*}
where $\delta= \left(1-e^{-c t\kappa} \right)^{\frac{1}{\ell}}$ and $C = (1-e^{-c t \kappa})^{-1-\frac{x_0}{\ell}}$.
\end{proof}

\begin{lemma}\label{lemma5}
Let $x_0 \ge 0$ and suppose $X$ satisfies 
\begin{align}\label{eq:representation342}
X(t) & = x_0 +  \ell Y \left(  \int_0^t \kappa(1+X(s) )ds  \right) 
\end{align}
where $Y$ is a unit-rate Poisson process and $\kappa>0$. For any integer $M  > x_0$, define $\tau_M =  \inf \{t\geq 0: X(t)  \geq M \}$. Then for $M \ge \ell + x_0$, 
\begin{equation}
\label{eq:87696796976}
P(\tau_M \leq t) \leq C \delta^M
\end{equation}
for some constant $C = C(t,\ell ,\kappa,x_0) >0$ and $\delta = \delta(t,\ell,\kappa,x_0) \in (0,1)$.
\end{lemma}
\begin{proof}
Define $Z(t) = X(t) + 1$, then we can rewrite \eqref{eq:representation342} to get the stochastic equation for $Z(t)$, which is
\begin{align}\label{eq:asdasfasfasfa}
Z(t) = (1+x_0) +  \ell Y \left(  \int_0^t \kappa Z(s) ds  \right). 
\end{align}
For any integer $k$, define $\mu_k= \inf \{t\geq 0: Z(t) \geq k \}$.  Then, by Lemma \ref{lemmalemma}, 
\begin{align}\label{eq:dfhsdf3}
P(\mu_k \leq t) \leq C \delta^k,
\end{align}
so long as  $k \ge \ell + x_0 + 1$. By the definition of $Z(t)$, $\tau_M = \mu_{M+1}$  and thus, so long as  $M +1 \ge \ell + x_0 +1$,
\begin{align*}
P(\tau_M \leq t ) = P(\mu_{M+1} \leq t) \leq C \delta^{M+1}  = \tilde{C} \delta^M
\end{align*}
where $\tilde{C} = C \delta = (1-e^{-c t \kappa})^{-1-\frac{x_0}{\ell}}$ and $\delta= \left(1-e^{-c t\kappa} \right)^{\frac{1}{\ell}}$ for $c = \max \{x_0+1,\ell\}$. 
\end{proof}

\begin{remark}\label{remark}
Suppose $\{Y_k\}_{k = 1}^K$ are independent unit-rate Poisson processes and suppose that the process $\overline{X}$ satisfies the representation
\begin{align*}
\overline{X}(t) & = x_0 + \sum_{k=1}^K   \ell Y_k \left(  \int_0^t \kappa_k (1+ \overline{X}(s))  ds \right).  
\end{align*}
Then the process has the same distribution as the process $X$ satisfying the representation
\begin{align*}
X(t) & = x_0 + \ell Y  \left(  \int_0^t \kappa (1+ X(s))  ds \right),  
\end{align*}
where $Y$ is a unit-rate Poisson process and $\kappa = \sum_{k=1}^K \kappa_k$.
By Lemma \ref{lemma5}, $X$ satisfies \eqref{probgeo}, implying that $\overline X$ does as well.   
\end{remark}

We now turn to the proof of Lemma \ref{boundednessAssumption}. 

\begin{proof}[Proof of Lemma \ref{boundednessAssumption}]  
Denote $\mathcal{P} = \{k\in \{1,\dots,K\}: \zeta_k\cdot \vec{1} >0\}$ and let $\ell = \displaystyle \max_{k \in \mathcal{P}} \{ \zeta_k\cdot  \vec{1} \}$.  We choose a specific $\theta \in \Theta$ and suppose that $X^\theta$ and $Z$ are defined (coupled)  via
\begin{align}
X^\theta(t) & = x_0+ \sum_{k=1}^K  Y_k \left( \int_0^t \lambda_k^\theta(X^\theta(s),s) ds \right) \zeta_k,\notag\\
Z(t) & =  \| x_0 \|_1 +  \ell \sum_{k\in \mathcal{P}}  Y_k \left(  \int_0^{t}  \overline C (1+Z(s)) ds \right)\label{dkajf;kjkjfa}
\end{align}
where $\{Y_k\}$ are independent unit-rate Poisson processes, and $\overline C$ is defined as in Assumption \ref{assumpMain}. Define $\mu_m  = \inf \{t\geq 0: Z(t) \geq m \}$. Then, so long as $m \ge \|x_0\|_1 + \ell$,  Lemma \ref{lemma5} (and Remark \ref{remark}) implies 
\begin{align}\label{eq:neededbound}
P(\mu_m \leq t ) \leq C \delta^m,
\end{align}
where  $c = \max \{ \|x_0\|_1+1,\ell\}$ and
\begin{equation}\label{eq:constconst}
C = (1-e^{-c t \overline CK})^{-1-\frac{\|x_0\|_1}{\ell}}, \qquad \qquad \delta= \left(1-e^{-c t \overline CK} \right)^{\frac{1}{\ell}}.
\end{equation}
 Turning to $X^\theta$, we take the 1-norm and find
\begin{align}
\| X^\theta(t) \|_1 = X^\theta(t) \cdot \vec{1} & = x_0\cdot \vec{1} + \sum_{k=1}^K Y_k \left( \int_0^t \lambda_k^\theta(X^\theta(s),s) ds \right) ( \zeta_k \cdot \vec{1})\notag\\
&\leq  \| x_0 \|_1 +  \ell \sum_{k\in \mathcal{P} }   Y_k \left(\overline C  \int_0^t  (1+\| X^\theta(s) \|_1)ds \right). \label{eq:u6897667855745}
\end{align}

\textcolor{black}{Note that if we can show $\|X^\theta(t) \|_1 \leq Z(t)$ for all $t\geq 0$, then $\tau_m^\theta\geq \mu_m$, and \eqref{probgeo} will be implied by \eqref{eq:neededbound}. Hence the rest of the proof will consist of showing $\|X^\theta(t) \|_1 \leq Z(t)$ for all $t\geq 0$.  We will proceed  by induction, and the arguments will rely crucially on the coupling in \eqref{dkajf;kjkjfa}. }

\textcolor{black}{Denote $\alpha_n$ as the $n$th jump time for $X^{\theta}$ due to reactions in $\mathcal{P}$.  That is,
\begin{align}
\alpha_n = \inf \{ t> \alpha_{n-1}: X^{\theta} (t) =  X^{\theta} (t-) + \zeta_k \text{ for some } k\in \mathcal{P} \},  \text{    with    } \alpha_0 = 0.
 \end{align}
\begin{enumerate}
\item Consider $n = 1$. Any reaction before time $\alpha_1$ will not increase $\| X^{\theta} (t) \|_1$, since $\zeta_k\cdot \vec{1} \leq 0$ for any $k \not \in \mathcal{P}$. On the other hand, $Z(t)$ is monotonically increasing, and hence $\|X^{\theta}(t)\|_1 \leq Z(t)$ for any $t< \alpha_{1}$. Consequently, 
\begin{align}
\|X^{\theta} (\alpha_1) \|_1 & = \|x_0\|_1 + \sum_{k=1}^K Y_k \left( \int_0^{\alpha_1} \lambda_k^\theta(X^\theta(s),s) ds \right) ( \zeta_k \cdot \vec{1})\notag\\
& \leq  \|x_0\|_1 + \ell \sum_{k\in \mathcal{P} }Y_k \left( \overline C  \int_0^{\alpha_1}  (1+ Z(s) )ds \right) = Z(\alpha_1).  \label{eq:sdfsdngioushg}
\end{align}
\item Now let $n \ge 2$ and suppose that $\|X^{\theta}(t)\|_1 \le Z(t)$ for all $t\leq \alpha_{n-1}$.  Then by the same reasoning as above, for any $t<\alpha_n$, $\| X^{\theta} (t) \|_1 \leq Z(t)$, and similarly to \eqref{eq:sdfsdngioushg}, $\|X^{\theta} (\alpha_n) \|_1 \leq Z(\alpha_n)$. 
\end{enumerate} 
Since $Z(t)$ is non-explosive and $\|X^{\theta} (t) \|_1 \leq Z(t)$,  we have that $\alpha_n \rightarrow \infty$, as $n\rightarrow \infty$, and the proof for the particular value of $\theta$ chosen is now complete. Noting that the bound in \eqref{eq:sdfsdngioushg} is uniform in $\theta$ completes the proof.
}
\end{proof}

\subsection{More Lemmas}\label{morelemmas}
In this section, we will prove some technical lemmas which will be used in Section \ref{proofmain} to prove the main result, Theorem \ref{mainthm1}.   The next two lemmas are provided for completeness.
\begin{lemma}\label{seriescon}
Suppose $k\geq 1 $ and $l\in \Z_{> 0}$ are fixed and $\delta \in (0,1)$. Then
\begin{align*}
\sum_{n =1}^\infty n^k \delta^{(n^{1/\ell})} < \infty.
\end{align*}
\end{lemma} 


\begin{lemma}\label{Calculus}
Suppose $x\in [0,1]$ and $n\geq 2$. Then 
\begin{align*}
1-(1-x)^n \leq nx. 
\end{align*}
\end{lemma}


To prove theorem \ref{mainthm1}, we will need some qualitative properties of $N_{q_K}^{\theta,\varepsilon}(t)$ introduced in \eqref{simulationrep} for the coupled processes. The final lemma concerns the moments of $N_{q_K}^{\theta,\varepsilon}(t)$.

\begin{lemma}\label{TCPMoments}
Let $( \mathcal{S},\mathcal{C}, \mathcal{R})$ be a reaction network with $d$ species and $K$ reactions.    Let  $\{ X^{\theta} \}$ be a family of stochastic models whose associated intensity functions $\{\lambda_k^{\theta}\}$ are parametrized by $\theta\in \Theta \subset \R_{>0}^\gamma$,  for some $\gamma \in \Z_{>0}$, where $\Theta$ is compact.  Assume $\{\lambda_k^\theta\}$ satisfies conditions \eqref{assump1} and \eqref{growth} of Assumption \ref{assumpMain}. Fix $\theta \in \Theta$ and $\theta+\varepsilon \in \Theta$, and suppose $(X^{\theta+\varepsilon}, X^{\theta})$ are coupled via the stacked coupling \eqref{stackedcoupling}, with $N_{q_K}^{\theta,\varepsilon}(t)$ defined as in \eqref{simulationrep}. Then for any $r\geq 1$,  $\EE[\left( N_{q_K}^{\theta,\varepsilon}(t) \right) ^r ] < \infty$. Consequently, $N_{q_K}^{\theta,\varepsilon}(t)< \infty$ almost surely. 
\end{lemma}

Note that there is no requirement for $\varepsilon$ to be small.  Instead, we just need both parameters to be contained within $\Theta$.

\begin{proof} 
We adopt the notation from Lemma \ref{boundednessAssumption}, where 
\[
	\tau_m^\theta =  \inf \left \{t\geq 0:  \| X^\theta(t)\|_1 = \sum_{i=1}^d X_i^\theta (t)  \geq m \right\},
\]
 and let $\displaystyle \tau_m = \min \{ \tau_m^\theta, \tau_m^{\theta+\varepsilon}\}$. Then by Lemma \ref{boundednessAssumption}, for $m$ large enough,
\begin{align*}
P(\tau_m \leq t) \leq P(\tau_m^\theta \leq t)  + P(\tau_m^{\theta+\varepsilon} \leq t)  \leq 2 C \delta^m. 
\end{align*}

The remainder of the proof focuses on bounding the tail probability $P(N_{q_K}^{\theta,\varepsilon}(t)\geq n)$, as $n\to \infty$. 
Define
\begin{align*}
n^* = n^{\frac{1}{p+1}}, 
\end{align*}
where $p$ is the order of the polynomial in condition (\ref{growth}) of Assumption \ref{assumpMain}. With this choice of $n^*$, by Lemma \ref{boundednessAssumption},
\begin{align}\label{eq:sd23rwfewef}
P(N_{q_K}^{\theta,\varepsilon}(t) \geq n, \tau_{n^*} \leq t) & \leq P(\tau_{n^*} \leq t) \leq 2C \delta^{n^\frac{1}{p+1}},
\end{align}
so long as $n$ is large enough.

If $\tau_{n^*} > t$, then $\|X(s)\|_1 \leq n^*$ for all $s\in [0,t]$, and so by condition (\ref{growth}) of Assumption \ref{assumpMain},
\begin{align*}
N_{q_K}^{\theta,\varepsilon}(t) \leq  \int_{[0,t] \times [0,\infty)} 1_{ \left[0 ,\sum_{k=1}^K \overline C (1+n^{\frac{p}{p+1}} )  \right)} (x) N(ds\times dx).
\end{align*}
Denote the process on right-hand side of the above equation by $Z_n(t)$ and note that it is a Poisson process with rate $\overline C K (1 + n^{\frac{p}{p+1}})$.  Since, $Z_n(t)$ is a Poisson random variable, we know
\begin{align*}
P(Z_n(t)\geq a ) \leq e^{ - a\log{\frac{a}{\EE[Z_n(t)]}} + a - \EE[Z_n(t)] }.
\end{align*}
so long as $a> \EE[Z_n(t)]$. Therefore, for $n$ large enough, 
\begin{align}
P(N_{q_K}^{\theta,\varepsilon}(t) \geq n, \tau_{n^*} >  t) 
&\leq P\left( Z_n(t) \geq n, \tau_{n^*} >  t  \right) \notag\\
&\leq P\left(Z_n(t) \geq n \right) \notag\\
&\leq \exp{ \left(  - n\log{\frac{n}{t\overline CK  (1+n^{\frac{p}{p+1}} ) }} + n - t\overline CK  (1+n^{\frac{p}{p+1}} ) \right)  }.\label{eq:asdaodsija23r}
\end{align}
Note that  the leading order on the exponent is $-n\log{n}$. 

Collecting \eqref{eq:sd23rwfewef} and \eqref{eq:asdaodsija23r}, we have
\begin{align*}
P(N_{q_K}^{\theta,\varepsilon}(t)\geq n) &= P(N_{q_K}^{\theta,\varepsilon}(t) \geq n, \tau_{n*} \le t) + P(N_{q_K}^{\theta,\varepsilon}(t) \geq n, \tau_{n*} > t)   \\
&\leq\exp{ \left(  - n\log{ \left (  \frac{n}{t\overline C K (1+n^{\frac{p}{p+1}} ) }\right)  } + n - t \overline C K  (1+n^{\frac{p}{p+1}} ) \right)  } +2 C \delta^{n^\frac{1}{p+1}}.
\end{align*}
We therefore have
\begin{align*}
\EE[ &\left(N_{q_K}^{\theta,\varepsilon}(t)\right)^r] = \sum_{n=1}^\infty n^r P(N_{q_K}^{\theta,\varepsilon}(t)  = n ) \leq \sum_{n=1}^\infty n^r P(N_{q_K}^{\theta,\varepsilon}(t)  \geq n ) \\
& \leq \sum_{n=1}^\infty n^r \left[ \exp{ \left(  - n\log{\left(\frac{n}{t \overline C K (1+n^{\frac{p}{p+1}} ) }\right)} + n - t \overline CK (1+n^{\frac{p}{p+1}} ) \right)  } + 2C \delta^{n^\frac{1}{p+1}} \right] \\
&\leq  \sum_{n=1}^\infty n^r \exp{ \left(  - n\log{\left( \frac{n}{t\overline CK  (1+n^{\frac{p}{p+1}} ) }\right)} + n - t \overline C K(1+n^{\frac{p}{p+1}} ) \right)  } + 2C \sum_{n=1}^\infty n^r \delta^{n^\frac{1}{p+1}}.
\end{align*}
The first series converges by the root test and second series converges by Lemma \ref{seriescon}. 
\end{proof}

\subsection{Proof of main result}\label{proofmain}

We restate the main theorem for completeness.\\

\noindent \textbf{Theorem 2.}  \textit{
Let $( \mathcal{S},\mathcal{C}, \mathcal{R})$ be a reaction network with $d$ species and $K$ reactions. Let  $\{ X^{\theta} \}$ be a family of stochastic models whose associated intensity functions $\{\lambda_k^{\theta}\}$ are parametrized by $\theta\in \Theta \subset \R_{>0}^\gamma$,  for some $\gamma \in \Z_{>0}$, where $\Theta$ is compact.  Suppose that Assumption \ref{assumpMain} holds and that $(X^{\theta}, X^{\theta+\varepsilon})$ are coupled via the stacked coupling (\ref{stackedcoupling}) with $X^{\theta}(0) = X^{\theta + \varepsilon}(0) = x_0 \in \Z^d_{\ge 0}$. Then, for any $\theta$ in the interior of  $\Theta$  and any  $r\ge 1$, there is a $C_{r,t}>0$  and $\overline \varepsilon \in \R_{> 0}$ so that 
\[
	\EE \left[\| X^{\theta+\varepsilon} (t) -X^{\theta}(t) \|_1^r \right] \leq C_{r,t} \| \varepsilon\|_1,
\]
when $ \| \varepsilon \|_1 \le \overline \varepsilon$.
}

\begin{proof}
Denote $\mathcal{P} = \{k\in \{1,\dots,K\}: \zeta_k\cdot \vec{1} >0\}$ and let $\hat \ell = \displaystyle \max_{k \in \{1,\dots,K\}} \{ |\zeta_k|\cdot  \vec{1} \}$.  

Fix $t>0$ and recall the counting process $N_{q_K}^{\theta,\varepsilon}(t)$ from \eqref{simulationrep}.  Denote by $\mu_i$ the time of the $i^{th}$ jump of $N_{q_K}^{\theta,\varepsilon}$. Next, let 
\begin{align*}
	\beta_{\theta,\varepsilon} & = \min_{i=1,2,..., N_{q_K}^{\theta,\varepsilon}(t)} \{ X^\theta (\mu_i)\neq X^{\theta+\varepsilon} (\mu_i) \},
\end{align*}
where the minimum of the emptyset is taken to be infinity.  Hence, $\beta_{\theta,\varepsilon}$ is the number of steps it took for $X^\theta$ and $X^{\theta + \varepsilon}$ to decouple, or is infinity if they have not decoupled by time $t$.

Since $N_{q_K}^{\theta,\varepsilon}(t) < \infty$ almost surely by Lemma \ref{TCPMoments}, the expectation can be calculated as
\begin{align}
\EE \big[\| X^{\theta+\varepsilon}& (t) -X^{\theta}(t) \|_1^r \big] = \sum_{n=1}^\infty \EE[\|X^{\theta+\varepsilon} (t) -X^{\theta}(t) \|_1^r 1_{\{N_{q_K}^{\theta,\varepsilon}(t) = n\}}]\notag \\
& = \sum_{n=1}^\infty \EE[\|X^{\theta+\varepsilon} (t) -X^{\theta}(t) \|_1^r 1_{\{N_{q_K}^{\theta,\varepsilon}(t) = n\}} 1_{\{\beta_{\theta,\varepsilon} \leq n\}}]  \leq   \sum_{n=1}^\infty (\hat \ell n)^r P(N_{q_K}^{\theta,\varepsilon}(t) = n, \beta_{\theta,\varepsilon} \leq \color{blue}{n} ) \notag\\
& =   \sum_{n=1}^\infty \hat \ell ^r n^r P(N_{q_K}^{\theta,\varepsilon}(t) = n)  P( \beta_{\theta,\varepsilon} \leq n | N_{q_K}^{\theta,\varepsilon}(t) = n).\label{eq:78969786967899}
\end{align}
\textcolor{black}{The last inequality is a consequence of Remark \ref{Re:stackedcoupling}, since $X^{\theta+\varepsilon} (0) =  X^{\theta} (0)$ and each jump recorded by either processes will increase their difference $\|X^{\theta+\varepsilon} - X^{\theta}\|_1$ by at most $\hat \ell$.  Since we are indicating on the fact that there was precisely $n$ jumps total, the total change in the 1-norm is at most $\hat \ell n$.}

We turn to the conditional probability above and begin by noting that
\begin{align}\label{eq:7896978678697}
P( \beta_{\theta,\varepsilon} > n | N_{q_K}^{\theta,\varepsilon}(t) = n) = \prod_{i =1}^nP( \beta_{\theta,\varepsilon} > i | N_{q_K}^{\theta,\varepsilon}(t) = n,  \beta_{\theta,\varepsilon} > i-1).
\end{align}
Next, for each $i \in \{1,\dots,n\}$, condition (\ref{coupling}) of Assumption \ref{assumpMain} gives
\begin{align*}
P( \beta_{\theta,\varepsilon} = i | N_{q_K}^{\theta,\varepsilon}(t) = n,  \beta_{\theta,\varepsilon} > i-1) &= \frac{\sum_{k=1}^K |\lambda_k^{\theta+\varepsilon}(X^\theta(\mu_{i-1}),\mu_i)-\lambda_k^{\theta}(X^\theta(\mu_{i-1}),\mu_i) |}{\sum_{k=1}^K \max\{\lambda_k^{\theta+\varepsilon}(X^\theta(\mu_{i-1}),\mu_i),\lambda_k^{\theta}(X^\theta(\mu_{i-1}),\mu_i)\}}\\
&\leq \overline C\|   \varepsilon \|_1.
\end{align*}
Hence,
\[
P( \beta_{\theta,\varepsilon} > i | N_{q_K}^{\theta,\varepsilon}(t) = n,  \beta_{\theta,\varepsilon} > i-1) \ge 1 - \overline C \|\varepsilon\|_1,
\]
and \eqref{eq:7896978678697} yields
\[
 P( \beta_{\theta,\varepsilon} > n | N_{q_K}^{\theta,\varepsilon}(t) = n) \ge (1 - \overline C\|\varepsilon\|_1)^n.
\]
Hence, choosing $\varepsilon$ so that $\overline C \|\varepsilon\|_1 < 1$, 
\begin{align*}
P( \beta_{\theta,\varepsilon} \leq n | N_{q_K}^{\theta,\varepsilon}(t) = n) \leq 1 - (1- \overline C\|\varepsilon\|_1)^n \leq \overline Cn\|\varepsilon\|_1
\end{align*}
where the last inequality comes from Lemma \ref{Calculus}. Plugging the above back into \eqref{eq:78969786967899} yields
\begin{align*}
\EE \left[\| X^{\theta+\varepsilon} (t) -X^{\theta}(t) \|_1^r \right] & \leq  \overline C \|\varepsilon\|_1 \sum_{n=1}^\infty \hat \ell^r n^{r+1} P(N_{q_K}^{\theta,\varepsilon}(t) = n)  = C \hat \ell^r \|\varepsilon\|_1\EE[\left( N_{q_K}^{\theta,\varepsilon}(t) \right) ^{r+1} ] 
\end{align*}
which is finite by Lemma \ref{TCPMoments}.
\end{proof}

The following corollary gives easy to check structural conditions on a mass-action network which guarantee the bound \eqref{eq:mainbound} holds.

\begin{corollary}
Let $( \mathcal{S},\mathcal{C}, \mathcal{R})$ be a reaction network with $d$ species and $K$ reactions. Let  $\{ X^{\theta} \}$ be a family of stochastic models whose associated intensity functions $\{\lambda_k^{\theta}\}$ are given by stochastic mass-action kinetics \eqref{stochasticMA}.  Suppose that $\theta = (\kappa_1, \dots, \kappa_k) \in \Theta \subset \R_{>0}^K$,  where $\Theta$ is compact.  Suppose  that only zeroth and first order reactions produce a net gain in molecules (i.e., those with $\zeta_k \cdot 1 > 0$) and that $(X^{\theta}, X^{\theta+\varepsilon})$ are coupled via the stacked coupling (\ref{stackedcoupling}) with $X^{\theta}(0) = X^{\theta + \varepsilon}(0) = x_0 \in \Z^d_{\ge 0}$. Then, for any $\theta$ in the interior of  $\Theta$  and any  $r\ge 1$, there is a $C_{r,t}>0$  and $\overline \varepsilon \in \R_{> 0}$ so that 
\[
	\EE \left[\| X^{\theta+\varepsilon} (t) -X^{\theta}(t) \|_1^r \right] \leq C_{r,t} \| \varepsilon\|_1,
\]
when $ \| \varepsilon \|_1 \le \overline \varepsilon$.
\end{corollary}
\begin{proof}
The proof is immediate by Lemma \ref{massactionassump} and Theorem \ref{mainthm1}. 
\end{proof}

\vspace{.4in}

\noindent {\Large \textbf{Acknowledgments}}\\

\noindent We gratefully acknowledge grant support from  the Army Research Office via grant  W911NF-18-1-0324.

\bibliographystyle{plain}
\bibliography{NonLip}

\end{document}